\documentclass{amsart}

\usepackage{graphicx} 
\usepackage[margin=3cm]{geometry}
\usepackage{fourier}
\usepackage{amssymb}
\usepackage{comment}
\usepackage{todonotes}
\PassOptionsToPackage{hyphens}{url}
\usepackage{hyperref}
\usepackage{xcolor}
\usepackage[inline]{enumitem}
\usepackage{ulem}

\newtheorem{theorem}{Theorem}
\newtheorem{proposition}[theorem]{Proposition}
\newtheorem{lemma}[theorem]{Lemma}
\newtheorem{corollary}[theorem]{Corollary}
\newtheorem{question}[theorem]{Question}
\newtheorem{questions}[theorem]{Questions}

\theoremstyle{remark}

\newtheorem{example}[theorem]{Example}
\newtheorem{examples}[theorem]{Examples}
\newtheorem{remark}[theorem]{Remark}
\newtheorem{remarks}[theorem]{Remarks}

\definecolor{blue-url}{RGB}{0,0,100}
\definecolor{red-url}{RGB}{100,0,0}
\definecolor{green-url}{RGB}{0,100,0}

\hypersetup{
	pdftitle={Isomorphism problems for ideal semigroups and power semigroups},
	pdfauthor={},
	pdfmenubar=false,
	pdffitwindow=true,
	pdfstartview=FitH,
	colorlinks=true,
	linkcolor=blue-url,
	citecolor=green-url,
	urlcolor=red-url
}

\newcommand{\evid}[1]{\textsf{#1}}
\newcommand{\red}{\mathrm{red}}
\newcommand{\fin}{\mathrm{fin}}

\providecommand\llb{\llbracket}
\providecommand\rrb{\rrbracket}

\renewcommand{\emptyset}{\varnothing}
\renewcommand{\setminus}{\smallsetminus}
\renewcommand{\,}{\kern 0.1em}

\makeatletter
\def\@bignumber#1#2{%
  \ifx#2\end
    #1\let\next\@gobble
  \else
    #1\hspace{0pt plus 1pt}\let\next\@bignumber
  \fi
  \next#2}
\newcommand{\bignumber}[1]{\@bignumber#1\end}
\makeatother

\title{Semigroups of ideals and isomorphism problems}

\author{Pedro A.~Garc\'ia-S\'anchez}
\address{(P.\,A.~Garc\'ia-S\'anchez) Departamento de \'Algebra and IMAG, Universidad de Granada | E-18071 Granada, Spain}
\email{pedro@ugr.es}
\urladdr{https://www.ugr.es/local/pedro}

\author{Salvatore Tringali}
\address{(S.~Tringali) School of Mathematical Sciences, Hebei Normal University | Shijiazhuang, Hebei province, 050024 China}
\email{salvo.tringali@gmail.com}
\urladdr{http://imsc.uni-graz.at/tringali}


\date{\today}

\begin{document}

\begin{abstract}
Let $H$ be a monoid (written multiplicatively). 
We call $H$ Archimedean if, for all $a, b \in H$ such that $b$ is a non-unit, there is an integer $k \ge 1$ with $b^k \in HaH$; strongly Archimedean if, for each $a \in H$, there is an integer $k \ge 1$ such that $HaH$ contains any product of any $k$ non-units of $H$; and duo if $aH = Ha$ for all $a \in H$. 

We prove that the ideals of two strongly Archimedean, cancellative, duo monoids make up isomorphic semigroups under the induced operation of setwise multiplication if and only if the monoids themselves are isomorphic up to units; and the same holds upon restriction to finitely generated ideals in Archimedean, cancellative, duo monoids. 
Then we use the previous results to tackle a new case of a problem of Tamura and Shafer from the late 1960s. 
\end{abstract}

\maketitle

\section{Introduction}
\label{sect:1}

Let $S$ be a semigroup. (We refer to Grillet's monograph \cite{Gri-95} for generalities on semigroups and monoids; and unless stated otherwise, we assume all semigroups are non-empty and write them multiplicatively.) Equipped with the (binary) op\-er\-a\-tion of setwise multiplication defined on the whole power set of $S$ by
$$
(X, Y) \mapsto XY := \{xy : x \in X,\, y \in Y\},
$$
the non-empty subsets of $S$ form a semigroup, herein denoted by $\mathcal P(S)$ and called the (\evid{large}) \evid{power semigroup} of $S$. Intensively studied by semigroup theorists and computer scientists through the 1980s and 1990s (partly due to their role in the theory of formal languages and automata), these types of structures offer a natural algebraic framework for old and new problems in additive combinatorics and related fields (see \cite{Bie-Ger-22, Tri-Yan2023(a), Tri-Yan2023(b)} and references therein); and they have played a pivotal role in the ongoing development of an arithmetic theory of semigroups and rings (see, for instance, \cite{Fan-Tri-2018, An-Tr-2021}, Section 4.2 in \cite{Tri-2022(a)}, and Example 4.5(3) in \cite{Cos-Tri-2023}) extending beyond the (somewhat narrow) boundaries of the classical theory \cite{GeHK06}. 

Power semigroups were first systematically investigated by Tamura and Shafer \cite{Tam-Sha1967} in the late 1960s (for historical remarks on power algebras, see \cite[Section 2]{Bri-1990}), raising the following:

\begin{question}\label{quest:1}
Suppose that a semigroup $H$ is \evid{globally isomorphic} to a semigroup $K$, meaning that $\mathcal P(H)$ is (semigroup-)isomorphic to $\mathcal P(K)$. Is it necessarily true that $H$ is isomorphic to $K$? 
\end{question}

The question was quickly answered in the negative by Mogiljanskaja \cite{Mog1973}, but remains widely open, among others, for \textit{finite} semigroups \cite[p.~5]{Ham-Nor2009}. More generally, one can ask:

\begin{question}\label{quest:2}
Given a class $\mathcal C$ of semigroups, prove or disprove that $H$ is globally isomorphic to $K$, for some $H, K \in \mathcal C$, if and only if $H$ is isomorphic to $K$.
\end{question}

It is easy to check (see, e.g., Remark 2.4(3) in \cite{Tri-2023(c)}) that if $f \colon H \to K$ is a semigroup iso\-mor\-phism, then the same holds for the function 
$F \colon \mathcal P(H) \to \mathcal P(K)$, $X \mapsto f[X]$, where $f[X] := \{f(x) : x \in X\}$ is the direct image of $X$ under $f$; in the terminology of \cite{Tri-Yan2023(a)}, we refer to $F$ as the \evid{augmentation} of $f$.
So, the main point of Question~\ref{quest:2} is the ``only if'' direction. For example, the answer to the question is in the affirmative for groups \cite{Sha67}, completely $0$-simple semi\-groups \cite[Theorem 5.9]{Tam-86}, and Clifford semi\-groups \cite[Theorem 4.7]{Gan-Zhao-2014}. 

The next question is a variant on the same theme that appears to have its origins in the work of Kobayashi \cite{Koba-84}, who lends part of the credit (or at least of the motivation) to Gould and Iskra \cite{Gou-Isk-1984}.

\begin{question}\label{quest:kobayashi}
Prove or disprove that, for all $H$ and $K$ in a given class of semigroups, every (semigroup) isomorphism $f$ from $\mathcal P(H)$ to $\mathcal P(K)$ restricts to an isomorphism from $H$ to $K$ (by which we mean that $f$ bijectively maps one-element subsets of $H$ to one-element subsets of $K$).
\end{question}

Obviously, a positive answer to Question \ref{quest:kobayashi} implies a positive answer to Question \ref{quest:2}. However, Kobayashi himself observed in \cite[Remark 1]{Koba-84} that the converse need not be true and proved in the same paper (see the unnumbered theorem on p.~218) that Question \ref{quest:kobayashi} has a positive answer within the class of semilattices.

There are many more problems in a similar spirit that are both of independent interest and helpful in dealing with Question~\ref{quest:2}. In particular, it is fairly clear that the non-empty \textit{finite} subsets of $S$ form a subsemigroup of $\mathcal P(S)$, herein denoted by $\mathcal P_\fin(S)$. We are thus prompted to the next questions.

\begin{questions}\label{quest:2fin}
Given a class $\mathcal C$ of semigroups, prove or disprove that, for all $H, K \in \mathcal C$, either of the following is true:
\begin{enumerate}[label=\textup{(\arabic{*})}]
\item\label{quest:2fin(1)} $\mathcal P_\fin(H)$ is isomorphic to $\mathcal P_\fin(K)$ if and only if $H$ is isomorphic to $K$.
\item\label{quest:2fin(2)} Every isomorphism from $\mathcal P_\fin(H)$ to $\mathcal P_\fin(K)$ restricts to an isomorphism from $H$ to $K$.
\item\label{quest:2fin(3)} Every isomorphism from $\mathcal P(H)$ to $\mathcal P(K)$ restricts to an isomorphism from $\mathcal P_\fin(H)$ to $\mathcal P_\fin(K)$.
\end{enumerate}
\end{questions}

In \cite[Theorem 3.2(3)]{Bie-Ger-22}, Bienvenu and Geroldinger answer Question~\ref{quest:2fin}\ref{quest:2fin(2)} in the affirmative when $\mathcal C$ is the class of \evid{numerical monoids} \cite{ns-ap}, i.e., the submonoids of $(\mathbb N, +)$ with finite complement in $\mathbb N$ (we use $\mathbb N$ for the set of non-negative integers and $\mathbb N^+$ for the positive integers).
Tringali has subsequently shown, as an instance of a more general result, that 
Questions~\ref{quest:kobayashi} and \ref{quest:2fin}\ref{quest:2fin(2)} have both a positive answer when $\mathcal C$ is the larger class of cancellative commutative semigroups \cite[Corollary 2.6]{Tri-2023(c)}. (We recall that the semigroup $S$ is \evid{cancellative} if $ax \ne ay$ and $xa \ne ya$ for all $a, x, y \in S$ with $x \ne y$.) 
The case of cancellative, but not necessarily commutative, semigroups is still open, and there is more to the story.

Let a (\evid{two-sided}) \evid{ideal} of the semigroup $S$ be a \textit{non-empty} set $I \subseteq S$ such that $IS \subseteq I$ and $SI \subseteq I$ (see~\cite[p.~16]{Gri-95}, where the empty set is, however, also considered to be an ideal). If $I$ and $J$ are ideals of $S$, then $(IJ)S = I(JS) \subseteq IJ$ and, similarly, $S(IJ) \subseteq IJ$. Therefore, the set of ideals of $S$ forms a subsemigroup of $\mathcal P(S)$, herein denoted by $\mathfrak I(S)$ and called the \evid{ideal semigroup} of $S$. This leads us to the next item on our problem list, which is a generalization of a question arisen during the Algebra and Number Theory seminar at the University of [\textit{omissis}] in Nov 2023, 
in the special case where $\mathcal C$ is the class of cancellative, finitely generated, commutative monoids with trivial group of units (the latter condition on the units is suggested by the observation that the ideal semigroup of \textit{any} group is trivial, see Remark \ref{rem:duo-makes-it-easier}\ref{rem:duo-makes-it-easier(2)}).

\begin{question}\label{quest:3}
Find (non-trivial) conditions for two semigroups to be isomorphic (if necessary, modulo a certain congruence) if and only if their ideal semigroups are isomorphic.
\end{question}

The interesting aspect of Question~\ref{quest:3} lies in the ``if'' direction. In fact, let $F$ be the augmentation of a semigroup isomorphism $f \colon H \to K$, and let $I \in \mathfrak I(H)$. By surjectivity of $f$, we have $K = f[H] = F(H)$. Since $F$ is a homomorphism $\mathcal P(H) \to \mathcal P(K)$ with the monotonicity property that $X \subseteq Y \subseteq H$ implies $F(X) = f[X] \subseteq f[Y] = F(Y)$, we find that $F(I) K = F(IH) \subseteq F(I)$ and, in an analogous way, $K F(I) \subseteq F(I)$. Namely, $F(I)$ is an ideal of $K$ and, hence (by applying the same argument to the functional inverse $f^{-1}$ of $f$), $F$ restricts to a semigroup isomorphism $\mathfrak I(H) \to \mathfrak I(K)$.

That being said, let $I$ be a non-empty subset of the semigroup $S$ and denote by $\widehat{S}$ a \evid{unitization} of $S$, i.e., the monoid obtained from $S$ by adjoining an identity element if needed (so that $\widehat{S} = S$ when $S$ is already a monoid). If $I = \widehat{S} X \widehat{S}$ for some (non-empty) $X \subseteq S$, then $IS = \widehat{S} X \widehat{S} S \subseteq \widehat{S} X \widehat{S} = I$ and similarly $SI \subseteq I$, i.e., $I$ is an ideal of $S$. Accordingly, we call $I$ a \evid{finitely generated} ideal of $S$ if there exists a finite subset $X$ of $S$ such that $I = \widehat{S} X \widehat{S}$. In particular, we define $I$ to be a \evid{principal ideal} of $S$ if $I = \widehat{S} x \widehat{S}$ for some $x \in S$, with the standard convention that braces are omitted whenever a singleton
appears within a product in a power semigroup. 

In general, the finitely generated ideals of $S$ do not make up a subsemigroup of $\mathfrak I(S)$, and neither do the principal ideals (cf.~Remark \ref{rem:duo-makes-it-easier}\ref{rem:duo-makes-it-easier(4)}).
Accordingly, we denote by $\mathfrak I_\fin(S)$ the subsemigroup of $\mathfrak I(S)$ generated by the finitely generated ideals of $S$, and by $\mathfrak P(S)$ the subsemigroup of $\mathfrak I_\fin(S)$ generated by the principal ideals (of course, every principal ideal is finitely generated). We thus arrive at the last questions on our list.

\begin{questions}\label{quest:5}
Find (non-trivial) conditions for the semigroups $H$ and $K$ to satisfy either of the following:
\begin{enumerate}[label=\textup{(\arabic{*})}]
\item\label{quest:5(1)} $\mathfrak I_\fin(H)$ is isomorphic to $\mathfrak I_\fin(K)$ if and only if $H$ is isomorphic to $K$.
\item\label{quest:5(2)} Every isomorphism from $\mathfrak I(H)$ to $\mathfrak I(K)$ restricts to an isomorphism from $\mathfrak P(H)$ to $\mathfrak P(K)$.
\item\label{quest:5(3)} Every isomorphism from $\mathfrak I_\fin(H)$ to $\mathfrak I_\fin(K)$ restricts to an isomorphism from $\mathfrak P(H)$ to $\mathfrak P(K)$.
\item\label{quest:5(4)} Every isomorphism from $\mathfrak I(H)$ to $\mathfrak I(K)$ restricts to an isomorphism from $\mathfrak I_\fin(H)$ to $\mathfrak I_\fin(K)$.
\item\label{quest:5(5)} Every isomorphism from $\mathcal P(H)$ to $\mathcal P(K)$ restricts to an isomorphism from $\mathfrak I(H)$ to $\mathfrak I(K)$.
\end{enumerate}
\end{questions}

Since a semigroup isomorphism is a bijection (on the level of the underlying sets), we gather from the above that the augmentation of a semigroup isomorphism $f \colon H \to K$ bijectively maps (non-empty) finite subsets of $H$ to finite subsets of $K$ and, hence, finitely generated ideals of $H$ to finitely generated ideals of $K$. So, once more, the essence of Questions~\ref{quest:2fin}\ref{quest:2fin(1)} and \ref{quest:5}\ref{quest:5(1)} lies in the ``only if'' part. 
Also, a positive answer to Question~\ref{quest:5}\ref{quest:5(2)} (resp., to Question \ref{quest:5}\ref{quest:5(3)}) implies a positive answer to Question \ref{quest:3} (resp., to Question \ref{quest:5}\ref{quest:5(1)}) assuming that $S$ is isomorphic to $\mathfrak P(S)$ for every $S \in \mathcal C$ (sufficient conditions for this to happen are provided in Lemma \ref{lem:pi-dc-general}). Finally, Question \ref{quest:5}\ref{quest:5(5)} establishes a connection between Questions~\ref{quest:3} and \ref{quest:kobayashi}, suggesting that, in addition to their intrinsic interest, isomorphism problems for semigroups of ideals could offer insights into isomorphism problems for power semigroups (see Section \ref{sec:back-to-power-sgrps} for an example of this principle in action).

With these preliminaries in place, we say that the semigroup $S$ is \evid{duo} if $aS = Sa$ for every $a \in S$;  \evid{Ar\-chi\-me\-de\-an} if, for all $a, b \in S$ such that $b$ is not a unit in $\widehat{S}$, there exists $k \in \mathbb N^+$ with $b^k \in \widehat{S} a \widehat{S}$; and \evid{strongly Ar\-chi\-me\-de\-an} if, for every $a \in S$, there exists $k \in \mathbb N^+$ such that the set $\widehat{S}a\widehat{S}$ contains any product $b_1 \cdots b_k$ of any $k$ non-units $b_1, \ldots, b_k \in \widehat{S}$. Clearly, every strongly Archimedean semigroup is Archimedean (let $b$ be an arbitrary non-unit in $\widehat{S}$ and choose $b_1 = \cdots = b_k = b$ in the last definition);  for a partial converse, see Proposition \ref{prop:arch-implies-strongly-arch}.

Our main goal is to give a positive answer to Question \ref{quest:5}\ref{quest:5(2)} for the class of strongly Ar\-chi\-me\-de\-an, cancellative, reduced, duo monoids (Theorem~\ref{th:isom-ideals-sp}), and to Question~\ref{quest:5}\ref{quest:5(3)} for the class of Ar\-chi\-me\-de\-an, cancellative, reduced, duo monoids (Theorem~\ref{thm:isom-ideals-fg-primary}). Here, a monoid is \evid{reduced} if its only unit is the identity. On the way, we obtain a characterization of strongly Archimedean, duo monoids in terms of a structural property of their ideals (Proposition \ref{prop:awesome-strongly-primary}); and of Archimedean, \textit{cancellative}, duo monoids in terms of a corresponding property of their finitely generated ideals (Proposition \ref{prop:primary-dc-fg-contains-pi}).

Archimedean and strongly Archimedean semigroups have been extensively studied in the mathematical literature, though mainly in the cancellative, commutative setting, where they are sometimes called primary and strongly primary, respectively. Rational semigroups in the sense of Sasaki and Tamura \cite{Sas-Tam-71} (i.e., additive subsemigroups of the non-negative cone of the ordered field of rational numbers) are Archimedean, cancellative, commutative semigroups (see the proof of Corollary \ref{cor:fin-generated-ideals-in-puiseux-monoids}); and numerical semigroups (i.e., additive subsemigroups of the non-negative integers with finite complement in $\mathbb N$) are strongly Archimedean (see the proof of Corollary \ref{cor:numerical-mons}). Note that some authors do not allow groups in the class of strongly Archimedean semigroups (see, for instance, ~\cite{sp-pm} and references therein). Nevertheless, according to our definitions, a group is strongly Archimedean in a vacuous way (for every element in a group is a unit).

As for duo semigroups (also known under the name of normalizing semigroups), they form an interesting class of semigroups halfway between the commutative and the non-commutative worlds. Introduced by Feller \cite{Fel-1958} in the context of rings (a ring is duo if its multiplicative semigroup is duo), they have been studied ever since by a number of authors, see, e.g., \cite{Brun-1975, Cou-1982, Yu-1995, Mar-2004, ger-osaka} and \cite[Section 5]{Cos-Tri-2023}.

With that said, let us note that any group is a strongly Archimedean, cancellative, duo monoid (in a trivial way). Moreover, there is no shortage of strongly Archimedean, cancellative, \textit{commutative} monoids in the literature. The constructions below offer non-trivial non-commutative examples (see also Example \ref{exa:more-examples}).

\begin{examples}\label{exa:strongly-archi-cance-non-commutative-duo}
\begin{enumerate*}[label=\textup{(\arabic{*})}, mode=unboxed]
\item\label{exa:strongly-archi-cance-non-commutative-duo(1)}
Let $R$ be a \evid{right discrete valuation domain}, i.e., a (commutative or non-commutative) domain with a non-unit $p \in R$ such that every non-zero element $a \in R$ can be written in
the form $p^n u$, where $n$ is a non-negative integer and $u \in R$ is a unit. Assume in addition that $pR^\times = R^\times p$, where $R^\times$ is the group of units of $R$; 
and let $R^\bullet$ be the multiplicative monoid of the non-zero elements of $R$. It is clear (by definition) that $R^\bullet = \bigcup_{n \in \mathbb N} p^n R^\times$, while a routine induction shows that $p^n R^\times = R^\times p^n$ and hence $(p^n R^\times)^k = p^{nk} R^\times$ for all $n, k \in \mathbb N$. This implies that $R^\bullet$ is a strongly Archimedean monoid, as the group of units of $R^\bullet$ is $R^\times$ and, for each $n \in \mathbb N$, we have $(R^\bullet \setminus R^\times)^{n+1} \subseteq p^n R^\bullet$. Moreover, $R^\bullet$ is duo, by the fact that, for every $m \in \mathbb N$ and $u \in R^\times$, 
\[
p^mu R^\bullet = p^m R^\bullet = \bigcup_{n \in \mathbb N} p^{m+n} R^\times = \bigcup_{n \in \mathbb N} R^\times p^{m+n} = R^\bullet p^m = R^\bullet p^m u.
\] 
\end{enumerate*}

\begin{enumerate*}[label=\textup{(\arabic{*})}, mode=unboxed, resume]
\item\label{exa:strongly-archi-cance-non-commutative-duo(2)}
Following \cite[Example (1.7)]{Lam-2001}, let $R$ be the ring of formal power series in one variable $x$ over a skew field $F$, with multiplication twisted by a ring automorphism $\sigma$ of $F$ in such a way that $ax = x\sigma(a)$ for every $a \in F$. We claim that $R$ satisfies all the conditions requested in item \ref{exa:strongly-archi-cance-non-commutative-duo(1)}.\\

\indent{}
To begin, we have from \cite[Exercise 19.7]{Lam-2003} that $R$ is a right discrete valuation domain. Moreover, a formal power series $f = \sum_{k \in \mathbb N} a_k x^k \in R$ is a unit if and only if $a_0$ is not the zero element $0_F$ of $F$. On the other hand, it is readily checked that $fx = xg$ and $xf = hx$, where $g := \sum_{k \in \mathbb N} \sigma(a_k) x^k$ and $h := \sum_{k \in \mathbb N} \sigma^{-1}(a_k) x^k$. Thus, we find that $R^\times x = x R^\times$ (since $a_0 \ne 0_F$ if and only if $\sigma(a_0) \ne 0_F$, if and only if $\sigma^{-1}(a_0) \ne 0_F$).\\

\indent{}Now, it is pretty obvious that $R^\bullet$ is a commutative monoid if and only if $\sigma$ is the identity map on $F$. So, for instance, we can make $R^\bullet$ non-commutative by choosing $F$ as a finite field of characteristic $p$ and order $q \neq p$, and $\sigma$ as the Frobenius automorphism of $F$.
\end{enumerate*}
\end{examples}

\section{Preparatory work and auxiliary results}

The problems listed in the introduction are phrased in the language of semigroups for the sake of generality and to facilitate comparison with analogous problems in the literature. From now on, our focus will shift on monoids (and more specifically on duo monoids), which are the central objects studied in the present work. Nevertheless, we will continue to formulate many ideas, concepts, and results in the parlance of semigroups. We begin with some preliminaries.

Let $H$ be a monoid. We write $1_H$ for the identity (element) of $H$ and $H^\times$ for its \evid{group of units}, i.e., the set of all $u \in H$ such that $uv = 1_H = vu$ for some $v \in H$, which is provably unique and hence denoted by $u^{-1}$. So, $H$ being reduced (see the third paragraph after Questions \ref{quest:5} in Section \ref{sect:1}) means that $H^\times = \{1_H\}$. It is a basic fact that $H^\times$ is a subgroup of $H$, and $u \in H$ is a unit if and only if $ux = yu = 1_H$ for some $x, y \in H$. We call $H$ a \evid{Dedekind-finite} monoid if $uv = 1_H$, for some $u, v \in H$, implies $vu = 1_H$ (or, equivalently, $uv \in H^\times$ implies $u, v \in H^\times$).

\begin{remark}\label{rem:morphisms}
Throughout, unless a statement to the contrary is made, a ``morphism'' will always be a \textit{semigroup} homomorphism. In particular, a \textit{monoid} homomorphism is a semigroup homomorphism from a monoid $H$ to a monoid $K$ that maps the identity of $H$ to the identity of $K$. It is easy to check that an isomorphism from a monoid to a monoid is, a fortiori, a monoid homomorphism, see, e.g., the last lines of \cite[Section 2]{Tri-2023(c)}.
\end{remark}

The presence of an identity simplifies many of the notions we have discussed so far. For future reference, we collect some simple observations related to this point in the next remarks.

\begin{remarks}
\label{rem:duo-makes-it-easier}
\begin{enumerate*}[label=\textup{(\arabic{*})}, mode=unboxed]
\item\label{rem:duo-makes-it-easier(1)}
A non-empty subset $I$ of a monoid $H$ is an ideal (of $H$) if and only if $IH = I = HI$; it suffices to consider that $I = I 1_H \subseteq IH$ and, similarly, $I \subseteq HI$. This shows that the ideal semigroup $\mathfrak I(H)$ of $H$ is a monoid with identity $H$ (rather than just a semigroup). Since $1_H H = H = H1_H$, it follows that $\mathfrak I_\fin(H)$ is a submonoid of $\mathfrak I(H)$, and $\mathfrak P(H)$ is a submonoid of $\mathfrak I_\fin(H)$. Likewise, the (large) power semigroup $\mathcal P(H)$ of $H$ is a monoid whose identity is the singleton $\{1_H\}$, and $\mathcal P_\fin(H)$ is a submonoid of $\mathcal P(H)$. Note, however, that $\mathfrak I(H)$ is not a submonoid of $\mathcal P(H)$ unless $H$ is a \evid{trivial} semigroup, namely, $H = \{1_H\}$.
\end{enumerate*}

\begin{enumerate*}[label=\textup{(\arabic{*})}, mode=unboxed, resume]
\item\label{rem:duo-makes-it-easier(2)} Let $I$ be an ideal of a monoid $H$. If $I$ contains a unit $u \in H$, then $I = H$. In fact, $1_H = \allowbreak u^{-1} u \in HI \subseteq I$ and hence $H = H1_H \subseteq HI \subseteq I$; it goes symmetrically with $IH$. Therefore, $I$ is a \evid{proper} ideal (i.e., $I$ is strictly contained in $H$) if and only if $I$ is disjoint from $H^\times$, which especially shows that the ideal semigroup of \textit{any} group is trivial (its only element being the group itself).
\end{enumerate*}

\begin{enumerate*}[label=\textup{(\arabic{*})}, mode=unboxed, resume]
\item\label{rem:duo-makes-it-easier(3)} Assume $H$ is a duo monoid. If $\emptyset \ne X \subseteq H$, then $XH = \bigcup_{x\in X}xH = \bigcup_{x\in H}Hx = HX$ and hence $HXH = XH^2 = XH = HX$. (We have $M \subseteq 1_M M \subseteq M^2 \subseteq M$ for any monoid $M$. We will frequently use this elementary observation without further mention later on.) It follows that a non-empty subset $I$ of $H$ is an ideal of $H$ if and only if $IH = I$. Moreover, every ideal of $H$ is of the form $XH$ for some non-empty set $X \subseteq H$, which is accordingly referred to as a \evid{generating set} of the ideal itself. Most notably, an ideal is finitely generated  if and only if it has a finite generating set; and is principal if and only if it is a set of the form $aH$ with $a \in H$.
\end{enumerate*}

\begin{enumerate*}[label=\textup{(\arabic{*})}, mode=unboxed, resume]
\item\label{rem:duo-makes-it-easier(4)}
Any unitization $\widehat{S}$ of a duo semigroup $S$ is a duo monoid. So, we have from item \ref{rem:duo-makes-it-easier(3)} that $(\widehat{S} X \widehat{S}) \allowbreak (\widehat{S} Y \widehat{S}) = \widehat{S} XY \widehat{S}$ for all non-empty $X, Y \subseteq S$. As a result, the setwise product of two finitely generated (resp., principal) ideals of $S$ is itself finitely generated (resp., principal). That is, an ideal $I$ of $S$ lies in $\mathfrak I_\fin(S)$ (resp., in $\mathfrak P(S)$) if and only if $I$ is a finitely generated (resp., principal) ideal.
\end{enumerate*}

\begin{enumerate*}[label=\textup{(\arabic{*})}, mode=unboxed, resume]
\item\label{rem:duo-makes-it-easier(5)}
It is clear from our definitions that a non-empty subset $I$ of a semigroup $S$ is an ideal of $S$ if and only if it is a \textit{proper} ideal of a unitization $\widehat{S}$ of $S$. Moreover, $S$ is duo if and only if $\widehat{S}$ is. Consequently, we gather from item \ref{rem:duo-makes-it-easier(3)} that, under the hypothesis of duoness, the ideals of $S$ are all and only the sets of the form $X\widehat{S}$ with $\emptyset \ne X \subseteq S$; most notably, the principal ideals of $S$ are the sets of the form $a \widehat{S}$ with $a \in S$.
\end{enumerate*}

\begin{enumerate*}[label=\textup{(\arabic{*})}, mode=unboxed, resume]
\item\label{rem:duo-makes-it-easier(6)}
Let $H$ be a duo monoid and suppose that $ab = 1_H$ for some $a, b \in H$. Since $aH = Ha$ and $bH = Hb$, there then exist $x, y \in H$ such that $xa = by = 1_H$. Thus, $a$ and $b$ are units, and $H$ is Dedekind-finite. It follows that if $H$ is not a group (i.e., $H \ne H^\times$), then $H \setminus H^\times$ is a proper ideal of $H$.
\end{enumerate*}
\end{remarks}

Given a semigroup $S$, we use the symbol $\mid_S$ for the \evid{divisibility preorder} on $S$, i.e., the binary relation defined by $a \mid_S b$, for some $a, b \in S$, if and only if $b \in \widehat{S}a\widehat{S}$, where $\widehat{S}$ is a unitization of $S$. 
We say that an element $a$ is \evid{associated} to an element $b$ in $S$ if $a \mid_S b$ and $b \mid_S a$ (or, equivalently, $\widehat{S}a\widehat{S} = \widehat{S}b\widehat{S}$). 
Being associated is clearly an equivalence, but need not be a (semigroup) congruence (see \cite[p.~10]{Gri-95} for terminology). Accordingly, we let the \evid{reduced quotient} of $S$, herein denoted by $S_\red$, be the quotient of $S$ by the smallest congruence containing the relation of associatedness.
It is a basic fact that the function $\pi \colon S \to S_\red$ sending an element $x \in S$ to its congruence class in the quotient $S_\red$ is a surjective homomorphism \cite[Proposition 3.2]{Gri-95}, called the \evid{canonical projection} of $S$ onto $S_\red$. The duo case is particularly smooth, as we are about to see.

\begin{proposition}\label{prop:red-duo-sgrp}
The following hold for a duo semigroup $S$:

\begin{enumerate}[label=\textup{(\arabic{*})}]
\item\label{prop:red-duo-sgrp(1)}
The relation of associatedness in $S$ is a congruence and the quotient $S_{\red}$ is itself a duo semigroup.
\item\label{prop:red-duo-sgrp(2)} If $S$ is cancellative, then so is $S_\red$.
\item\label{prop:red-duo-sgrp(3)} The reduced quotient of any unitization of $S$ is a reduced monoid.
\item\label{prop:red-duo-sgrp(4)} $S_\red$ is (semigroup-)isomorphic to $\mathfrak{P}(S)$.
\end{enumerate}
\end{proposition}

\begin{proof}
Denote by $\simeq_S$ the relation of associatedness in $S$. We will repeatedly use, without further comment, that the unitization of a duo semigroup is a duo monoid, and thus, by  Remark \ref{rem:duo-makes-it-easier}\ref{rem:duo-makes-it-easier(3)}, two elements $a, b \in S$ are associated (that is, $a \simeq_S b$) if and only if $a\widehat{S}  = b\widehat{S}$, where $\widehat{S}$ is a unitization of $S$.

\ref{prop:red-duo-sgrp(1)} To begin, assume $a \simeq_S c$ and $b \simeq_S d$, that is, $a\widehat{S} = c\widehat{S}$ and $b\widehat{S} = d\widehat{S}$. It follows that $
ab\widehat{S} = ad\widehat{S} = a\widehat{S}d = b\widehat{S}d = bd\widehat{S}$, i.e., $ab \simeq_S cd$. Consequently, $\simeq_S$ is a congruence on $S$ and $S_\red$ is the quotient of $S$ by $\simeq_S$. 

Next, let $\pi$ be the canonical projection of $S$ onto $S_\red$, and let $x \in S$. Since $S$ is duo, we have that, for each $y \in S$, there exists $z \in S$ such that $xy = zx$ and hence $\pi(x) \pi(y) = \pi(xy) = \pi(zx) = \pi(z) \pi(x)$. This ultimately means that $\pi(x) S_\red \subseteq S_\red \pi(x)$, and the reverse inclusion is proved symmetrically. So, $S_\red$ is duo.

\ref{prop:red-duo-sgrp(2)} Suppose $S$ is cancellative and fix $a, b, c \in S$. Since $\pi$ is a homomorphism, it is clear that $\pi(a)\pi(b) = \allowbreak \pi(a) \allowbreak \pi(c)$ if and only if $ab \simeq_S ac$ (by our definitions, two elements of $S$ have the same image under $\pi$ if and only if they are associated), if and only if $ab\widehat{S} = ac\widehat{S}$. So, noting that $\widehat{S}$ is cancellative (by the hypothesis that $S$ is), it is immediate (see \cite[Remark 2.1]{Tri-2023(c)}) that $b\widehat{S} = c\widehat{S}$ and hence $\pi(b) = \pi(c)$. Symmetrically, $\pi(b)\pi(a) = \allowbreak \pi(c) \allowbreak \pi(a)$ if and only if $\pi(b) = \pi(c)$. Therefore, we conclude that $S_\red$ is cancellative.

\ref{prop:red-duo-sgrp(3)} 
We may assume without loss of generality that $S$ is a monoid (and thus $S = \widehat{S}$). Then, $S_\red$ is itself a monoid (rather than just a semigroup), its identity being the class $\pi(1_S)$ of the identity $1_S$ of $S$. In fact, we have $\pi(1_S) \pi(x) = \pi(x) = \pi(x) \pi(1_S)$ for all $x \in S$.

It remains to check that $S_\red$ is reduced. Suppose that $\pi(a) \pi(b) = \pi(1_S)$ for some $a, b \in S$, and hence $abS = \allowbreak S$. Since $ab \in bS$ (by the duoness of $S$), there exist $x, y \in S$ such that $1_S = ax = by$. It follows that $S \subseteq axS \subseteq aS \subseteq S$ (i.e., $S = aS$) and, similarly, $S = bS$. This means that $a \simeq_S 1_S \simeq_S b$, implying $\pi(a) = \pi(1_S) = \pi(b)$ and ultimately proving that $S_\red$ is indeed a reduced monoid.

\ref{prop:red-duo-sgrp(4)} By Remark \ref{rem:duo-makes-it-easier}\ref{rem:duo-makes-it-easier(5)}, the principal ideals of $S$ are the sets of the form $a\widehat{S}$ with $a \in S$. So, by Remark \ref{rem:duo-makes-it-easier}\ref{rem:duo-makes-it-easier(4)}, the function $f \colon S \to \mathfrak{P}(S), a \mapsto a\widehat{S}$ is well-defined and surjective. Moreover, $f$ is a semigroup homomorphism $S \to \mathfrak P(S)$, because $
f(ab) = ab\widehat{S} = ab\widehat{S}^{\,2} = a\widehat{S} b \widehat{S} = f(a) f(b)$ for all $a, b \in S$.
On the other hand, the set
\[
\ker(f) := \{(a,b) \in S \times S: f(a) = f(b)\} = \{(a,b) \in S \times S: a\widehat{S} = b \widehat{S}\}
\]
is nothing but the (graph of the) relation of associatedness on $S$. Hence, we conclude from \cite[Proposition 3.3(3)]{Gri-95} that $S_\red$ (i.e., the quotient of $S$ by $\simeq_S$) is isomorphic to $\mathfrak P(S)$ (i.e., the image of $f$).
\end{proof}

Before proceeding further, we show in the upcoming remark how the notion of associatedness employed in this work aligns with the classical notion found in the literature on cancellative commutative monoids; and we use this observation to work out a construction (Example \ref{exa:more-examples}) that will be useful towards the end of Section \ref{sec:back-to-power-sgrps}.

\begin{remark}
\label{rem:associated-in-unit-canc-duos}
Let $H$ be a duo monoid that satisfies $xy \ne x \ne yx$ for all $x, y \in H$ with $y \notin H^\times$; such a property is commonly referred to as \evid{unit-cancellativity}, and it is evidently verified by any cancellative monoid. We claim that two elements $a, b \in H$ are associated if and only if $b \in aH^\times$, if and only if $b \in H^\times a$ (in particular, this gives $aH^\times = H^\times a$). By symmetry, it is enough to prove the first biconditional.

It is clear that if $a = bu$ for some $u \in H^\times$, then $b = au^{-1}$ and hence $a \simeq_H b$. So, let us assume $a \simeq_H b$. By Remark \ref{rem:duo-makes-it-easier}\ref{rem:duo-makes-it-easier(3)}, this implies $aH = bH$, i.e., there exist $c, d \in S$ with $ac = b$ and $a = bd$. It follows that $a = \allowbreak bd = \allowbreak acd$, which, by the unit-cancellativity of $H$ and Remark \ref{rem:duo-makes-it-easier}\ref{rem:duo-makes-it-easier(3)}, shows  that $c, d \in H^\times$ (as wished).

Accordingly, we get from Proposition~\ref{prop:red-duo-sgrp}\ref{prop:red-duo-sgrp(4)} that if $S$ is a duo semigroup and its unitization $\widehat{S}$ is reduced and unit-cancellative (as a monoid), then two elements in $S$ are associated if and only if they are equal, with the result that $S$ is isomorphic to the reduced quotient $S_\red$ via the map $x \mapsto \{x\}$.
\end{remark}

\begin{example}\label{exa:more-examples}
Let $H$ be a unit-cancellative duo monoid and fix $a \in H$. We claim $aK = Ka$, where $K := H \setminus H^\times$. It will follow that the non-units of $H$ form a unit-cancellative duo semigroup (which, in particular, is the empty semigroup when $H = H^\times$), as we know from Remark \ref{rem:duo-makes-it-easier}\ref{rem:duo-makes-it-easier(6)} that $HK \subseteq K$.

By symmetry, it suffices to show that $aK \subseteq Ka$. For, let $x \in K$. Since $H$ is duo, there exists $y \in H$ such that $ax = ya$. We need to check that $y$ is a non-unit. Suppose to the contrary that $y$ is a unit. Then, from Remark \ref{rem:associated-in-unit-canc-duos}, we gather that $ax = az$ for some $z \in H^\times$, which implies $axz^{-1} = a$ and hence, by the unit-cancellativity of $H$, $x \in H^\times$. This is, however, absurd (because $x$ is a non-unit of $H$) and proves our claim.

Now, suppose $a \in K$ and assume that, in addition to being unit-cancellative and duo, $H$ is also Archimedean (resp., strongly Ar\-chi\-me\-de\-an). By definition and Remark \ref{rem:duo-makes-it-easier}\ref{rem:duo-makes-it-easier(3)}, there is an integer $n \ge 1$ such that $b^n \in aH$ for every non-unit $b \in H$ (resp., $b_1 \cdots b_n \in aH$ for all non-units $b_1, \ldots, b_n \in H$). So, recalling from the above that $HK \subseteq K$, we obtain that $b^{n+1} \in \allowbreak aHb \subseteq aK$ for every $b \in K$ (resp., $b_1 \cdots b_{n+1} \in aHb_{n+1} \subseteq aK$ for all $b_1, \ldots, b_{n+1} \in K$). In other words, $K$ is an Archimedean (resp., strongly Archimedean) semigroup without identity, and it is therefore immediate to check that any unitization of $K$ is an Archimedean (resp., strongly Archimedean), unit-cancellative, duo monoid with trivial group of units. Furthermore, it is obvious that $K$ is cancellative if and only if $H$ is.
\end{example}

A \evid{divisor-closed} subsemigroup of a semigroup $S$ is a subsemigroup $T$ of $S$ such that if $x \mid_S y$ and $y \in T$, then $x \in T$; in this case, we also say that the subsemigroup $T$ is divisor-closed in $S$. Divisor-closed sub\-monoids are defined in an analogous way. Any divisor-closed subsemigroup of a monoid $H$ contains the group of units $H^\times$ and hence the identity $1_H$, making it a sub\-monoid (rather than just a subsemigroup). So, for instance, a monoid is Dedekind-finite if and only if its group of units is a divisor-closed submonoid.

We denote by $\llb X \rrb_S$ the intersection of all divisor-closed subsemigroups of the semigroup $S$ that contain a set $X \subseteq S$, and for each $x \in S$ we write $\llb x \rrb_S$ in place of $\llb \{x\} \rrb_S$. It is routine that $\llb X\rrb_S$ is itself a divisor-closed subsemigroup of $S$ (containing $X$). Moreover, divisor-closedness is preserved under isomorphisms:

\begin{lemma}\label{lem:dc-under-isomorphism}
Let $f \colon M\to N$ be a semigroup isomorphism. If $H$ is a divisor-closed subsemigroup of $M$, then $f[H]$ is a divisor-closed subsemigroup of $N$.
\end{lemma}

\begin{proof}
We only check that $f[H]$ is a divisor-closed subsemigroup of $N$, for it is a basic fact that a sem\-i\-group homomorphism maps subsemigroups to subsemigroups (see, e.g., Proposition 3.3(1) in \cite{Gri-95}). 
    
Assume $ab = f(h) \in f[H]$ for some $a, b \in N$ and $h \in H$. Since $f$ is surjective, there exists $a, b \in M$ such that $a = f(x)$ and $b = f(y)$. It follows that $f(xy) = f(x) f(y) = ab = f(h)$, implying by the injectivity of $f$ that $xy = h \in H$. Consequently, we obtain by the divisor-closedness of $H$ in $M$ that $x, y \in H$. But this means that $a, b \in f[H]$, showing that $f[H]$ is divisor-closed in $N$. 
\end{proof}

The next lemma provides an iterative construction of divisor-closed subsemigroups and will come in handy in the proof of Lemma \ref{lem:els-principal-div-closed} (we could not find any reference in the literature, so we include a proof).

\begin{lemma}
\label{lem:desc-div-closed-submonoid}
Let $X$ be a non-empty subset of a semigroup $S$. Define $D_0(X) := X$ and, for each $n \in \mathbb N$, let $D_{n+1}(X)$  be the subsemigroup of $S$ generated by the set $\{x \in S: x \mid_S y, \text{ for some }y \in D_n(X)\}$. Then, we have
\begin{equation}\label{equ:divisor-closed-subsgrp}
    \llb X\rrb_S=\bigcup_{n\in \mathbb{N}} D_n(X).
\end{equation}
\end{lemma}

\begin{proof} 
Set $N := \bigcup_{n \in \mathbb{N}} D_n(X)$. It is fairly obvious that $D_0(X) = X \subseteq \llbracket X \rrbracket_S$. If, on the other hand, $D_n(X) \subseteq \llbracket X \rrbracket_S$ for some $n \in \mathbb N$, then the set $\Lambda_n := \{x \in S: x \mid_S y, \text{ for some } y \in D_n(X)\}$ is a subset of $\llbracket X \rrbracket_S$ (because every $y \in \allowbreak D_n(X)$ is an element of $\llbracket X \rrbracket_S$ and $\llbracket X \rrbracket_S$ is divisor-closed in $S$). So, $\llbracket X \rrbracket_S$ contains the subsemigroup of $S$ generated by $\Lambda_n$ (by the fact that $\llbracket X \rrbracket_S$ contains $\Lambda_n$ and is itself a subsemigroup), that is, $\llbracket X \rrbracket_S$ contains $D_{n+1}(X)$. Therefore, we see by induction that $X \subseteq N \subseteq \llbracket X \rrbracket_S$. It remains to prove that $\llbracket X \rrbracket_S \subseteq N$; and to this end, it is enough to establish that $N$ is a divisor-closed subsemigroup of $S$, as $\llbracket X \rrbracket_S$ is the intersection of \textit{all} divisor-closed subsemigroups containing $X$.

To begin, let $a, b \in N$. There then exist $i, j \in \mathbb{N}$ such that $a \in D_i(X)$ and $b \in D_j(X)$. Putting $m := \max\{i, j\}$ and considering that $D_k(X) \subseteq D_{k+1}(X)$ for all $k \in \mathbb N$, it follows that $a, b \in D_m(X)$. Since $D_m(X)$ is a subsemigroup of $S$, this implies $ab \in D_m(X) \subseteq N$ and ultimately shows that $N$ is itself a subsemigroup.

Next, let $x \in N$. Then, $x \in D_k(X)$ for some $k \in \mathbb N$. It is thus clear from our definitions that $y \mid_S x$ implies $y \in D_{k+1}(X) \subseteq N$, and we conclude from here that $N$ is divisor-closed in $S$.
\end{proof}

Roughly speaking, our goal in the present work is to prove that, in some circumstances, one can recover a monoid $H$, modulo units and up to isomorphism, from either $\mathfrak I(H)$ or $\mathfrak I_\fin(H)$. Essential to this end is the following lemma, which plays a key role in the proofs of our main results (Theorems \ref{th:isom-ideals-sp} and \ref{thm:isom-ideals-fg-primary}).

\begin{lemma}\label{lem:pi-dc-general}
If $S$ is a cancellative, duo semigroup, then $\mathfrak{P}(S)$ is a divisor-closed subsemigroup of $\mathfrak{I}(S)$.
\end{lemma}

\begin{proof}
We have already noted (and more than once) that the set $\mathfrak P(S)$ of principal ideals
of $S$ is a subsemigroup of the ideal semigroup $\mathfrak I(S)$. So, we only need to check that $\mathfrak{P}(S)$ is divisor-closed in $\mathfrak{I}(S)$. 

For, fix a unitization $\widehat{S}$ of $S$, let $I, J \in \mathfrak{I}(S)$, and assume $IJ \in \mathfrak{P}(S)$. By Remark \ref{rem:duo-makes-it-easier}\ref{rem:duo-makes-it-easier(5)}, this means that  
$IJ = a\widehat{S} = \{a\} \cup aS$ for some $a \in S$, implying that there exist $x \in I$ and $y \in J$ with $a = xy$. 
Let $z \in I$. Then, $zy \in a\widehat{S}$ and hence $zy = ab = xyb$ for a certain $b \in \widehat{S}$. As $S$ is duo, there exists $c \in S$ such that $yb = cy$. It follows that $zy = xcy$, and cancelling out $y$, we obtain $z = xb \in x\widehat{S}$. Therefore, $I = x\widehat{S} \in \mathfrak P(S)$ and, in a similar way, $J \in \mathfrak P(S)$. So, $\mathfrak P(S)$ is a divisor-closed subsemigroup of $\mathfrak I(S)$. 
\end{proof}

\section{Main results}
\label{sec:3}

In this section, we prove the main results of the paper (Theorems \ref{th:isom-ideals-sp} and \ref{thm:isom-ideals-fg-primary}). 
We start with a characterization (Proposition \ref{prop:awesome-strongly-primary}) of strongly Archimedean, duo monoids that may be of independent interest. The reader may want to review Remark \ref{rem:duo-makes-it-easier}\ref{rem:duo-makes-it-easier(1)} before proceeding.

\begin{lemma}\label{lem:els-principal-div-closed}
The following hold for any two ideals $I$ and $J$ of a semigroup $S$:
\begin{enumerate}[label=\textup{(\arabic{*})}]
\item\label{lem:els-principal-div-closed(1)} If $I$ divides $J$ in the ideal semigroup $\mathfrak I(S)$ of $S$, then $J \subseteq I$.
\item\label{lem:els-principal-div-closed(2)} If $J\in \llb I\rrb_{\mathfrak{I}(S)}$, then $I^n\subseteq J$ for some $n \in \mathbb N^+$.
\end{enumerate}
\end{lemma}

\begin{proof}
\ref{lem:els-principal-div-closed(1)} If $I$ divides $J$ in $\mathfrak I(S)$, then either $J = I$ or there exist $H, K \in \mathfrak I(S)$ such that $J = HI$, $J = IK$, or $J = \allowbreak HIK$. In any case, we have $J \subseteq I$, because $HI \subseteq SI \subseteq I$ and $IK \subseteq IS \subseteq I$ (by definition of an ideal).

\ref{lem:els-principal-div-closed(2)} Fix $k \in \mathbb N$ and let $D_k(I)$ be defined as in Lemma~\ref{lem:desc-div-closed-submonoid} (with $\mathfrak I(S)$ in place of $S$ and $I$ in place of $X$). In view of Eq.~\eqref{equ:divisor-closed-subsgrp}, it suffices to show that, for each $J\in D_k(I)$, there is an integer $n \ge 1$ such that $I^n\subseteq J$. 

We argue by induction on $k$. The case $k = 0$ is trivial, as $J\in D_0(I)$ implies $I=J$, allowing us to choose $n=1$. So, suppose the statement holds for $k$ and let us prove it for $k+1$. To this end, let $J\in D_{k+1}(I)$. By the recursive definition of $D_{k+1}(I)$, there exist $r \in \mathbb N^+$, $J_1,\ldots,J_r\in \mathfrak{I}(S)$, and $I_1, \ldots, I_r \in D_k(I)$ such that $J=J_1\cdots J_r$ and $J_1 \mid_{\mathfrak{I}(S)} I_1$, \ldots, $J_r \mid_{\mathfrak{I}(S)} I_r$. Consequently, we gather from item \ref{lem:els-principal-div-closed(1)} that, for each $i \in \{1, \ldots, r\}$, $I_i \subseteq J_i$ and hence, by the inductive hypothesis, $I^{k_i} \subseteq I_i$ for a certain $k_i \in \mathbb N^+$. It follows that $I^n \subseteq I_1\cdots I_r\subseteq J_1\cdots J_r=J$, where $n := k_1+\dots+k_r \in \mathbb N^+$.
\end{proof}

\begin{proposition}
\label{prop:awesome-strongly-primary}
The following conditions are equivalent for a duo monoid $H$:
\begin{enumerate}[label=\textup{(\alph{*})}]
    \item\label{prop:awesome-dc(1)}  $H$ is strongly Archimedean.
    \item\label{prop:awesome-dc(2)} Any non-trivial divisor-closed subsemigroup of $\mathfrak{I}(H)$ contains the set $\mathfrak P(H)$ of principal ideals.
\end{enumerate}
\end{proposition}

\begin{proof}
Suppose first that $H$ is strongly Archimedean and let $a \in H$. Then, by definition, there is an integer $n \ge 1$ such that $(H \setminus H^\times)^n \subseteq aH$. Let $M$ be a non-trivial divisor-closed subsemigroup of $\mathfrak{I}(H)$, and let $I \in M$ be a \textit{proper} ideal of $H$. By Remark \ref{rem:duo-makes-it-easier}\ref{rem:duo-makes-it-easier(2)}, we are guaranteed that $I \subseteq H \setminus H^\times$. It follows that $I^n \subseteq aH$ and hence $I^n=aX$ for some $X\subseteq H$. Since $H^2 = H$ and, by Remark \ref{rem:duo-makes-it-easier}\ref{rem:duo-makes-it-easier(3)}, $XH = HX$, we thus find that
\[
M \ni I^n = I^n H = aXH = a(XH)H = a(HX)H = (aH)(XH).
\]
So, considering that $XH$ is an ideal of $H$ and $M$ is divisor-closed in $\mathfrak I(H)$, we obtain that $aH$ is an element of $M$, which ultimately proves that $M$ contains the set of principal ideals of $H$.

For the other implication, assume $H$ is not a group (otherwise the conclusion is obvious) and let $a \in H$. We know from Remark~\ref{rem:duo-makes-it-easier}\ref{rem:duo-makes-it-easier(6)} that $H \setminus H^\times$ is a proper ideal of $H$. Hence, the principal ideal $aH$ is, by hypothesis, an element of  $\llb H \setminus H^\times\rrb_{\mathfrak I(H)}$. It follows by Lemma~\ref{lem:els-principal-div-closed} that $(H \setminus H^\times)^n \subseteq aH$ for some integer $n \ge 1$. Consequently, $H$ is a strongly Archimedean monoid.
\end{proof}

We are now ready to affirmatively answer Question \ref{quest:5}\ref{quest:5(2)} for the class of strongly Archimedean, cancellative, duo monoids. As a by-product, we will also provide a positive answer to Question \ref{quest:3} for the same class.

\begin{theorem}
\label{th:isom-ideals-sp}
If $H$ and $K$ are strongly Archimedean, cancellative, duo monoids, then every isomorphism from $\mathfrak{I}(H)$ to $\mathfrak I(K)$ restricts to an isomorphism from $\mathfrak P(H)$ to $\mathfrak P(K)$. In particular, if $\mathfrak{I}(H)$ is isomorphic to $\mathfrak I(K)$, then $H_{\red}$ is isomorphic to $K_{\red}$.
\end{theorem}

\begin{proof}
Let $f$ be an isomorphism $\mathfrak{I}(H) \to \mathfrak{I}(K)$. By Lemma \ref{lem:pi-dc-general} and Prop\-o\-si\-tion~\ref{prop:awesome-strongly-primary}, $\mathfrak{P}(H)$ and $\mathfrak{P}(K)$ are the intersection of all divisor-closed subsemigroups of $\mathfrak{I}(H)$ and $\mathfrak{I}(K)$, respectively. On the other hand, since the functional inverse $f^{-1}$ of $f$ is an isomorphism $\mathfrak I(K) \to \mathfrak I(H)$, we have from Lemmas
\ref{lem:pi-dc-general} and \ref{lem:dc-under-isomorphism} that $f[\mathfrak{P}(H)]$ is a divisor-closed subsemigroup of $\mathfrak{I}(K)$ and $f^{-1}[\mathfrak{P}(K)]$ is a divisor-closed subsemigroup of $\mathfrak I(H)$. It follows that $\mathfrak{P}(K) \subseteq f[\mathfrak{P}(H)]$ and $\mathfrak{P}(H) \subseteq f^{-1}[\mathfrak{P}(K)]$, which implies $\mathfrak{P}(K) \subseteq \allowbreak f[\mathfrak{P}(H)] \subseteq \mathfrak{P}(K)$. This finishes the proof, as the ``In particular'' part follows directly from Proposition~\ref{prop:red-duo-sgrp}\ref{prop:red-duo-sgrp(4)}.
\end{proof}

In the remainder of the section, we focus on finitely generated ideals. We begin with a technical result borrowed from \cite{Cos-Tri-2023} and then establish a lemma that will prove useful in two different ways.

\begin{lemma}\label{lem:cos-tri-2023}
Let $H$ be a duo monoid, and let $m, n \in \mathbb N^+$. Then, $Hx_1H \cdots Hx_nH \subseteq x_{\sigma(1)} \cdots  x_{\sigma(m)} H$ for all $x_1, \allowbreak \ldots, \allowbreak x_n \in H$ and every (strictly) increasing function $\sigma \colon \{1, \ldots, m\} \to \{1, \ldots, n\}$.
\end{lemma}

\begin{proof}
This is a special case of \cite[Lemma 5.7]{Cos-Tri-2023}, using that $aH = Ha$ for all $a \in H$ (by the duoness of $H$).
\end{proof}

\begin{lemma}\label{lem:primary-power-fg-included-pi}
    Let $H$ be an Archimedean, duo monoid and $I$ be a finitely generated ideal of $H$ with $I \ne H$. Then, for all $a \in H$, there exists an integer $k \ge 1$ such that $I^k\subseteq aH$. 
\end{lemma}

\begin{proof}
Since $I$ is finitely generated, it holds that $I=XH$ for some non-empty finite set $X \subseteq H$; and since $X \subseteq I\neq H$, we gather from Remark \ref{rem:duo-makes-it-easier}\ref{rem:duo-makes-it-easier(2)} that $X \subseteq H \setminus H^\times$. Let $a \in H$ and write $X = \{x_1, \ldots, x_n\}$, where $n := |X|$. 
   
As $H$ is Archimedean, we have that, for every $i \in \{1, \ldots, n\}$, there is an integer $k_i \ge 1$ such that $a \mid_H x_i^{k_i}$. Set $k := k_1 + \cdots + k_n - n + \allowbreak 1$ and fix $x \in X^k$. We have that $x = x_{i_1}\cdots x_{i_k}$ for some $i_1, \ldots, i_k \in \{1, \ldots, n\}$. Let $e_i$ be, for each $i \in \{1, \ldots, n\}$, the number of indices $j \in \{1, \ldots, k\}$ such that $x_i = x_{i_j}$. It is clear that $e_1+\dots+e_n=k$. So, $e_r \ge k_r$ for at least one index $r \in \{1, \ldots, n\}$, or else $e_1 + \cdots + \allowbreak e_n \le k_1 + \cdots + \allowbreak k_n - n < k$ (a contradiction). On the other hand, we get from Lemma \ref{lem:cos-tri-2023} that $x = x_r^{e_r} y$ for a certain $y \in H$. It follows that $a\mid_H x_r^{k_r} \mid_H x_r^{e_r} \mid_H x$ and hence $X^k\subseteq aH$. So, by Remark \ref{rem:duo-makes-it-easier}\ref{rem:duo-makes-it-easier(3)} and a routine induction, $I^k = (XH)^k = X^k H \subseteq aH$.
\end{proof}

It is natural to ask if, under reasonable hypotheses, an Archimedean monoid is strongly Archimedean. A sufficient condition for this to happen is provided by the following proposition, where we say after \cite[Definition 4.1]{Cos-Tri-2023} that a monoid is \evid{finitely generated up to units} if there is a finite set $A \subseteq H$ such that every $x \in \allowbreak H$ factors as a (finite) product of elements in $H^\times A H^\times$. For instance, the monoid $R^\bullet$ considered in Example \ref{exa:strongly-archi-cance-non-commutative-duo} is finitely generated up to units (in addition to being strongly Archimedean, cancellative, and duo).

\begin{proposition}\label{prop:arch-implies-strongly-arch}
A finitely-generated-up-to-units, duo monoid is Archimedean if and only if it is strongly Archimedean. 
\end{proposition}

\begin{proof}
Let $H$ be an Archimedean, duo monoid (the ``if'' direction is true in general, as already noted in the last part of Section \ref{sect:1}). If $H$ is a group, then we have nothing to do. Otherwise, $H$ being Dedekind-finite (by Remark \ref{rem:duo-makes-it-easier}\ref{rem:duo-makes-it-easier(6)}) and finitely generated up to units (by hypothesis) implies, by \cite[Remarks 2.1(1) and 4.4(1)]{Cos-Tri-2023}, that there is a non-empty \textit{finite} subset $A$ of $H$ such that $H \setminus H^\times = Y \cup Y^2 \cup \cdots$, where the right-hand side is the subsemigroup of $H$ generated by $Y := H^\times A H^\times \subseteq H$. On the other hand, we have from Remark \ref{rem:associated-in-unit-canc-duos} that $AH^\times = H^\times A H^\times = H^\times A$ (note that $H^\times H^\times = H^\times$). It follows (by a routine induction) that
$
Y^n = AH^\times Y^{n-1}$ for all $n \in \mathbb N^+$, where $Y^0 = \{1_H\}$.
So, considering that $H^\times Y = Y$, we obtain 
\[
H \setminus H^\times = AH^\times (Y^0 \cup Y \cup Y^2 \cup \cdots) = AH^\times (Y^0 \cup (H \setminus H^\times))  = A (H^\times \cup (H \setminus H^\times)) = AH.
\]
Therefore, $H \setminus H^\times$ is a finitely generated ideal (because $A$ is, by construction, a finite subset of $H$); and by Lemma \ref{lem:primary-power-fg-included-pi}, we conclude that $H$ is strongly Archimedean.
\end{proof}

The next result is an Archimedean version of the characterization proven in Proposition \ref{prop:awesome-strongly-primary}. A remarkable difference between the two is that the proof of Proposition \ref{prop:awesome-strongly-primary} does not depend on cancellativity.

\begin{proposition}\label{prop:primary-dc-fg-contains-pi}
The following hold for a duo monoid $H$:
\begin{enumerate}[label=\textup{(\arabic{*})}]
    \item\label{prop:primary-dc-fg-contains-pi(1)} If $H$ is Archimedean, then any non-trivial divisor-closed submonoid of $\mathfrak{I}_{\fin}(H)$ contains $\mathfrak{P}(H)$.
    \item\label{prop:primary-dc-fg-contains-pi(2)} The previous implication can be reversed under the additional hypothesis that $H$ is cancellative.
    \end{enumerate}
\end{proposition}

\begin{proof}
\ref{prop:primary-dc-fg-contains-pi(1)} Assume $H$ is Archimedean and let $M$ be a non-trivial divisor-closed subsemigroup of $\mathfrak I_\fin(H)$. Then, fix $a \in H$ and $I \in M$ with $I \ne H$. By Lemma~\ref{lem:primary-power-fg-included-pi}, there is an integer $k \ge 1$ such that $I^k \subseteq \allowbreak aH$. On the other hand, $I^k$ being a finitely generated ideal of $H$ implies by Remark \ref{rem:duo-makes-it-easier}\ref{rem:duo-makes-it-easier(3)} that $I^k = YH$ for some finite $Y \subseteq H$. So, $Y = Y1_H \subseteq YH \subseteq aH$ and hence $Y = aX$ for a certain $X \subseteq H$. In fact, we can choose $X$ to be \textit{finite} (by the finiteness of $Y$). Again by Remark \ref{rem:duo-makes-it-easier}\ref{rem:duo-makes-it-easier(3)}, it follows that 
\[
M \ni I^k=(aX)H=a(XH)H=a(HX)H=(aH)(XH),
\]
where each of $aH$ and $XH$ is itself a finitely generated ideal of $H$. Since $M$ is divisor-closed in $\mathfrak I_\fin(H)$ and $aH$ is a divisor of $I^k$ in $\mathfrak I_\fin(H)$, we can thus conclude that $aH \in M$ and hence $\mathfrak P(H) \subseteq M$.

\ref{prop:primary-dc-fg-contains-pi(2)} Suppose $H$ is cancellative, and let $a \in H$ and $b \in H \setminus H^\times$. By Lemma \ref{lem:pi-dc-general}, $\mathfrak{P}(H)$ is a divisor-closed subsemigroup of $\mathfrak{I}(H)$. Since $bH$ is a principal ideal of $H$ and $\llb bH \rrb_{\mathfrak{I}(H)}$ is, by definition, the intersection of all divisor-closed subsemigroups of $\mathfrak I(H)$ containing $bH$, it is thus clear that 
    \[
    \llb bH \rrb_{\mathfrak{I}(H)}\subseteq \mathfrak{P}(H)\subseteq \mathfrak{I}_\fin(H).
    \]
    So, $\llb bH \rrb_{\mathfrak{I}(H)}$ is a divisor-closed subsemigroup of $\mathfrak I_\fin(H)$, which implies, by hypothesis, that $\llb bH \rrb_{\mathfrak{I}(H)}$ contains $aH$. It follows, by Lemma~\ref{lem:els-principal-div-closed}\ref{lem:els-principal-div-closed(2)} and a routine induction, that $b^nH = (bH)^n \subseteq aH$ for some $n \in \mathbb N^+$. In particular, we have $b^n \in aH$, which ultimately proves that $H$ is Archimedean. 
\end{proof}

We are finally ready to present an Archimedean analogue of Theorem \ref{th:isom-ideals-sp}.

\begin{theorem}\label{thm:isom-ideals-fg-primary}
If $H$ and $K$ are Archimedean, cancellative, duo monoids, then every isomorphism from $\mathfrak{I}_{\fin}(H)$ to $\mathfrak{I}_{\fin}(K)$ restricts to an isomorphism from $\mathfrak{P}(H)$ to $\mathfrak{P}(K)$. In particular, if $\mathfrak{I}_\fin(H)$ is isomorphic to $\mathfrak I_\fin(K)$, then $H_{\red}$ is isomorphic to $K_{\red}$.
\end{theorem}

\begin{proof}
The existence of an isomorphism from $\mathfrak P(H)$ to $\mathfrak P(K)$ is proved in pretty much the same way as we did with  Theorem~\ref{th:isom-ideals-sp}, the only difference being that here we use Proposition~\ref{prop:primary-dc-fg-contains-pi} in place of Proposition~\ref{prop:awesome-strongly-primary}. The ``In particular'' part of the statement is then straightforward from Proposition~\ref{prop:red-duo-sgrp}\ref{prop:red-duo-sgrp(4)}.
\end{proof}

For the next corollaries, we recall that a (\evid{rational}) \evid{Puiseux monoid} is a submonoid of the non-negative rational numbers under addition; in particular, every numerical monoid is a Puiseux monoid. In recent years, these structures have been the subject of intensive research and proved extremely useful in the construction of examples and counterexamples in commutative algebra and related fields, see \cite{Got-Li-2023} and references therein.

\begin{corollary}\label{cor:fin-generated-ideals-in-puiseux-monoids}
If $H$ and $K$ are Puiseux monoids with $\mathfrak I_\fin(H)$ isomorphic to $\mathfrak I_\fin(K)$, then $H$ is isomorphic to $K$.
\end{corollary}

\begin{proof}
Let $M$ be a Puiseux monoid (written additively). Since Puiseux monoids are cancellative, reduced and commutative (thus duo), it is enough by Remark \ref{rem:associated-in-unit-canc-duos} and Theorem~\ref{thm:isom-ideals-fg-primary} to check that $M$ is Archimedean.

Let $x, y \in M$ with $y \ne 0$, and write $x = a/b$ and $y = c/d$, where $a \in \mathbb N$ and $b, c, d \in \mathbb N^+$. We have $bcx = \allowbreak ady = ac$ and hence $x \mid_M ky$, where $k := ad \in \mathbb N^+$.
Consequently, $M$ is Ar\-chi\-me\-de\-an, since $y$ is in fact an arbitrary non-unit of $M$ (obviously, every Puiseux monoid is reduced).
\end{proof}

By \cite[Example~
3.7)]{sp-pm}, not every Puiseux monoid is strongly Archimedean (see also \cite[Theorem~3.4]{sp-pm}). So, we cannot apply Theorem~\ref{th:isom-ideals-sp} to the class of Puiseux monoids, for which the answer to Question~\ref{quest:3} remains open. However, Theorem~\ref{th:isom-ideals-sp} applies to numerical monoids, leading to the last result of the section.

\begin{corollary}\label{cor:numerical-mons}
If $H$ and $K$ are numerical monoids with $\mathfrak{I}(H)$ isomorphic to $\mathfrak{I}(K)$, then $H=K$.
\end{corollary}

\begin{proof} Every numerical monoid is cancellative, commutative, and reduced; in addition, it is strongly Ar\-chi\-me\-de\-an, for it contains all positive integers from a certain point on. So, by Remark \ref{rem:associated-in-unit-canc-duos} and Theorem \ref{thm:isom-ideals-fg-primary}, $H$ is isomorphic to $K$. This finishes the proof, as it is folklore that two numerical monoids are isomorphic if and only if they are equal \cite[Theorem 3.2]{Hig-969}.
\end{proof}

\section{Back to the Tamura--Shafer problem}
\label{sec:back-to-power-sgrps}

This short section can be regarded as more of an appendix. The conclusions may not be particularly deep, as they focus on very specific classes of objects. Nevertheless, it is at least intriguing that, as we are about to illustrate, some of the ideas set forward in the present work can be applied to isomorphism problems for power semigroups, which are notoriously difficult.

\begin{lemma}\label{lem:iso-image-contains-one-old}
Let $H$ and $K$ be monoids, and let $f$ be an isomorphism $\mathcal{P}(H)\to \mathcal{P}(K)$. If $1_H \in f^{-1}(K)$, then $f[\mathfrak{I}(H)] \subseteq \mathfrak{I}(K)$. 
If, in addition, $1_K \in f(H)$, then $f$ restricts to an isomorphism $\mathfrak{I}(H) \to \mathfrak{I}(K)$.
\end{lemma}

\begin{proof}
Assume $1_H \in f^{-1}(K) =: M$ and let $I \in \mathfrak I(H)$. Then, by Remark \ref{rem:duo-makes-it-easier}\ref{rem:duo-makes-it-easier(1)}, $I = I1_H \subseteq IM \subseteq IH = I$ and hence $I = IM$. It follows that $f(I)K = f(I)f(M) = f(IM) = f(I)$; and in a similar way, $f(I) = K f(I)$. Thus, $f$ sends ideals of $H$ to ideals of $K$, implying that $f[\mathfrak{I}(H)] \subseteq \mathfrak{I}(K)$. This is enough to finish the proof, because the same argument applied to the inverse of $f$ (which is an isomorphism from $\mathcal P(K)$ to $\mathcal P(H)$) shows that if $1_K \in f(H)$, then $f^{-1}[\mathfrak I(K)] \subseteq \mathfrak I(H)$ and hence $\mathfrak I(K) \subseteq f[\mathfrak I(H)]$.
\end{proof}

For the next lemma, we recall that a semigroup is \evid{periodic} if each of its elements generates a finite subsemigroup; in a group, this is equivalent to every element having finite order.

\begin{lemma}\label{lem:idemp-monoid}
    Let $H$ be a strongly Archimedean, cancellative, duo  monoid, and let $M$ be a non-empty subset of $H$ such that $M = M^2$. If the group of units $H^\times$ of $H$ is  periodic, then $1_H \in M$.
\end{lemma}

\begin{proof}
Since $M = M^2$ (by hypothesis), a routine induction shows that $M = M^k$ for all $k \in \mathbb N^+$ (namely, $M$ is a subsemigroup of $H$). We will freely use this observation in the remainder of the proof.

We claim that $H^\times$ (the group of units of $H$) and $M$ are not disjoint. Suppose to the contrary that $M \subseteq H \setminus H^\times$, and let $x \in M$. Since $H$ is strongly Archimedean and duo, there then exists an integer $n \ge 1$ such that $(H \setminus H^\times)^n \subseteq x^2 H$. It follows that $M = M^n \subseteq x^2 H$ and, hence, $x = x^2 y$ for a certain $y \in H$. So, by cancelling out $x$, we obtain $1_H = xy$, which, by Remark \ref{rem:duo-makes-it-easier}\ref{rem:duo-makes-it-easier(6)}, yields $x \in M \cap H^\times$ (a contradiction). 

Thus, the set $M \cap H^\times$ is non-empty. Accordingly, let $u \in M \cap H^\times$. Then, since $H^\times$ is a group and each of its elements has finite order (by hypothesis), there is an integer $k \ge 1$ such that $1_H = u^k \in M^k = M$. 
\end{proof}

Our final lemma is essentially contained in \cite{Sha67}; we include it here only for the sake of completeness.

\begin{lemma}\label{lem:globally-isomorphism-have-isomorphic-unit-groups}
Let $H$ and $K$ be globally isomorphic monoids. The group of units of $H$ is then isomorphic to the group of units of $K$. In particular, if either of $H$ or $K$ is reduced, then both are.
\end{lemma}

\begin{proof}
The power semigroup of a monoid $M$ is a monoid in its own right, the identity of $\mathcal P(M)$ being the singleton $\{1_M\}$. On the other hand, a non-empty set $X \subseteq M$ is a unit in $\mathcal P(M)$ if and only if $X$ is a one-element subset of the group of units $M^\times$ of $M$; in particular, if $X \in \mathcal P(M)^\times$ and hence $XY = YX = \{1_M\}$ for some $Y \in \mathcal P(M)$, then $xy = yx = 1_M$ for all $x \in X$ and $y \in Y$, implying that $X, Y \subseteq M^\times$ and hence $|X| \le |XY| = 1$. It follows that $M^\times$ is isomorphic to the group of units of $\mathcal P(M)$ via the embedding $u \mapsto \{u\}$. 

That being said, let $f$ be an isomorphism from $\mathcal P(H)$ to $\mathcal P(K)$. Since any semigroup isomorphism from a monoid to a monoid sends the identity to the identity (see, e.g., the last lines of \cite[Section 2]{Tri-2023(c)}), $f$ restricts to an isomorphism between the corresponding groups of units. So, we conclude from the above that $H^\times$ is isomorphic to $K^\times$. The rest is obvious.
\end{proof}

So we arrive at the last theorem of the paper. The result provides a positive answer to Question~\ref{quest:2} (i.e., the Tamura--Shafer problem) for a new, albeit rather special, class of \textit{cancellative} semigroups for which it appears that no answer was previously known;
and it serves as a complement to \cite[Theorem 25]{Tri-2023(c)}, which solves Question \ref{quest:kobayashi} (i.e., the Kobayashi problem) in the affirmative for the class of cancellative \textit{commutative} semigroups. 

\begin{theorem}\label{thm:glob-iso-sp}
 If $H$ and $K$ are strongly Archimedean, cancellative, duo monoids and at least one of them is reduced, then every isomorphism from $\mathcal{P}(H)$ to $\mathcal{P}(K)$ restricts to an isomorphism from $\mathfrak{I}(H)$ to $\mathfrak{I}(K)$. In particular, if $H$ and $K$ are globally isomorphic, then they are isomorphic.
\end{theorem}

\begin{proof}
Note first that, by Lemma \ref{lem:globally-isomorphism-have-isomorphic-unit-groups}, $H$ and $K$ are both reduced monoids. The conclusion is then immediate from Theorem \ref{th:isom-ideals-sp} and Lemmas~\ref{lem:iso-image-contains-one-old} and \ref{lem:idemp-monoid}.
\end{proof}

It is worth noting that, by Examples \ref{exa:strongly-archi-cance-non-commutative-duo}\ref{exa:strongly-archi-cance-non-commutative-duo(2)} and \ref{exa:more-examples}, a strongly Ar\-chi\-me\-de\-an, cancellative, reduced, duo monoid need not be commutative. Therefore, Theorem \ref{thm:glob-iso-sp} is not a special case of \cite[Theorem 25]{Tri-2023(c)}.

\section*{Acknowledgments}

The first author was partly supported by the grants ProyExcel\_00868 and FQM--343, both funded by the Junta de Andaluc\'ia. He also acknowledges financial support from grant PID2022-138906NB-C21, funded by MICIU/AEI/\bignumber{10.13039}/\bignumber{501100011033} and by the ERDF ``A way of making Europe''; and from the Spanish Ministry of Science and Innovation (MICINN), through the ``Severo Ochoa and María de Maeztu Programme for Centres and Unities of Excellence'' (CEX2020-001105-M). The second author was supported by the Natural Science Foundation of Hebei Province, grant no.~A2023205045. Both authors are grateful to the anonymous referee of an earlier version of this work for their careful proofreading and valuable suggestions.

\end{document}